\documentclass[pdflatex,sn-mathphys-num]{sn-jnl}

\usepackage{amsmath,amssymb,amsfonts}%
\usepackage{amsthm}%
\usepackage{mathrsfs}%
\usepackage[title]{appendix}%
\usepackage{xcolor}%
\usepackage{textcomp}%
\usepackage{manyfoot}%
\usepackage{booktabs}%
\usepackage{algorithm}%
\usepackage{algorithmicx}%
\usepackage{algpseudocode}%
\usepackage{listings}%
\usepackage{cleveref}%
\usepackage{xfrac}%

\newtheorem{theorem}{Theorem}
\newtheorem{example}{Example}%
\newtheorem{remark}{Remark}%
\newtheorem{lemma}{Lemma}%
\newtheorem{corollary}{Corollary}

\begin{document}

\title[A Nehari manifold method for nonvariational problems]{A Nehari manifold method for nonvariational problems}








\author[1,2]{\fnm{Radu} \sur{Precup}}\email{radu.precup@ubbcluj.ro}

\author*[2,3]{\fnm{Andrei} \sur{Stan}}\email{andrei.stan@ubbcluj.ro}
\equalcont{These authors contributed equally to this work.}
\affil*[1]{ \orgname{Faculty of Mathematics and Computer Science, Babeș-Bolyai
University}, \orgaddress{ \city{Cluj-Napoca}, \postcode{400084}, \country{Romania}}}

\affil[2]{ \orgname{Institute of Advanced Studies in Science and Technology, Babeş-Bolyai University}, \orgaddress{ \city{Cluj-Napoca}, \postcode{400110}, \country{Romania}}}

\affil[3]{ \orgname{Tiberiu Popoviciu
Institute of Numerical Analysis, Romanian Academy}, \orgaddress{ \city{Cluj-Napoca}, \postcode{400110}, \country{Romania}}}


\abstract{The aim of this paper is to extend the Nehari manifold method from the variational setting to the nonvariational framework of fixed point equations. This is achieved  by constructing a radial energy functional that generalizes the standard one from the variational case. Furthermore, the solutions obtained through our method are localized in conical annular sets, which leads to the existence of multiple solutions. The abstract results are illustrated by two representative applications.
}




\maketitle

\section{Introduction and preliminaries}

The Nehari manifold method is a well-known technique for finding nonzero
critical points of functionals. Given a $C^{1}$ functional $E$ defined on a
Banach space $X,$ the Nehari manifold is defined as:%
\begin{equation*}
\mathcal{N}=\left\{ u\in X\setminus \left\{ 0\right\} :\ \left\langle
E^{\prime }\left( u\right) ,u\right\rangle _{X}=0\right\} ,
\end{equation*}%
where $\left\langle \cdot ,\cdot \right\rangle _{X}$ denotes the duality
between $X$ and its dual $X^{\ast }.$ It is clear that any nonzero critical
point of $E$ lies in $\mathcal{N}.$ Therefore, one may expect that it is
easier to find nonzero critical points by restricting the search to this
smaller set, as fewer or weaker conditions may be required compared to
working in the entire space $X.$ This is the case, for example, with the
Palais-Smale compactness condition. Once the Nehari manifold is defined, the
method proceeds by minimizing $E$ on $\mathcal{N}.$ This approach is
effective when $E$ has a specific geometric structure, namely, for each
direction $v\in X,$ $\left\vert v\right\vert =1,$ the function
\begin{equation*}
0<t\ \mapsto \ E\left( tv\right)
\end{equation*}%
has a unique critical point at some $t_{v}>0.$ In this case, the derivative
satisfies
\begin{equation*}
\frac{d}{dt}\left. E\left( tv\right) \right\vert _{t=t_{v}}=\left\langle
E^{\prime }\left( t_{v}v\right) ,v\right\rangle _{X}=0,
\end{equation*}%
so $u=t_{v}v\in $ $\mathcal{N}.$ Conversely, if $u\in \mathcal{N},$ then
\begin{equation*}
0=\left\langle E^{\prime }\left( u\right) ,u\right\rangle _{X}=\left\langle
E^{\prime }\left( \left\vert u\right\vert \frac{u}{\left\vert u\right\vert }%
\right) ,\frac{u}{\left\vert u\right\vert }\right\rangle _{X},
\end{equation*}%
showing that $t=\left\vert u\right\vert $ is a critical point of the
function $E\left( t\frac{u}{\left\vert u\right\vert }\right) ,$ i.e., $t_{%
\frac{u}{\left\vert u\right\vert }}=\left\vert u\right\vert .$ Thus, under
this geometric condition, the Nehari manifold can be characterized as%
\begin{equation*}
\mathcal{N}=\left\{ t_{v}v:\ v\in X,\ \left\vert v\right\vert =1\right\} .
\end{equation*}
This method  has received increasing attention in recent years and has been successfully applied to various types of  problems~\cite{nehari_opt, nehari_parabolic1, faraci-silva, nehari_parabolic2, n1}.
In the recent papers \cite{a,pa}, this technique has been refined to locate critical points
within cones, which is useful when searching for solutions with additional
properties such as positivity. Also, in \cite{pa2}, the method has been combined with Schauder's fixed point theorem to obtain a localized solution for a system. 


Compared to other techniques in critical point theory---such as general
linking methods, and particularly the mountain pass theorem---the Nehari
manifold method essentially requires only a radial behavior of the energy
functional. Motivated by this observation, in this paper we propose a
similar approach within the framework of fixed point theory, specifically
for equations of non-variational type. To extend the technique to a fixed
point equation of the form $u=T\left( u\right) ,$ we begin by introducing a
three-variable functional $\mathcal{F}\left( u,v,w\right) $ which is used to
define a radial energy type function $\mathcal{E}$ in any direction $v.$
That is, for any unit vector $v$ and $t>0,$ we define:%
\begin{equation*}
\mathcal{E}\left( v\right) \left( t\right) =\int_{0}^{t}\mathcal{F}\left(
T\left( sv\right) ,sv,v\right) ds.
\end{equation*}%
For each direction $v,$ the derivative of $\mathcal{E}\left( v\right) $ with
respect to $t$ is given by%
\begin{equation*}
\frac{d}{dt}\mathcal{E}\left( v\right) \left( t\right) =\mathcal{F}\left(
T\left( tv\right) ,tv,v\right) .
\end{equation*}%
This expression suggests that,
for some direction $v$, the quantity
$$t\mapsto\mathcal{F}\left( T\left( tv\right)
,tv,v\right) $$ can be interpreted 
as the  (radial) potential of the energy
functional $\mathcal{E}(v)$ at the moment $t$. This interpretation is natural since, as shown in the paper, 
in the particular case of equations
possessing a variational structure, i.e., the operator $T$ is of potential type in the sense that $T=I-E'$, for some functional $E$, then the value of  $\mathcal{E}(v)(t)$ corresponds to $E(tv)$. In this situation, the classical Nehari manifold \( \mathcal{N} \) is recovered, and existing methods apply.
We
emphasize that even in this case, our results significantly refine the
classical ones by localizing the critical point and by imposing a weaker
condition on the radial geometry of the energy function, which, in our case,
requires only a unique maximum point in each direction rather than a unique
critical point which is a maximum.


Under suitable conditions on the operator $T$ and the functional $\mathcal{F}
$, and through a surprising combination with  the Birkhoff--Kellogg invariant-direction theorem, we
develop a Nehari-type method for establishing the existence of fixed points
of $T$ within a conical annular region of a Banach space.



To summarize, the novelties of this paper are the following:
\begin{itemize}
    \item[(a):]First and most importantly, we extend the Nehari manifold method from the variational setting to a nonvariational framework. This is achieved by constructing a radial energy that corresponds to the usual energy in the variational case. 
    \item[(b):] The solutions obtained by this method are localized in conical annular sets. By applying the method repeatedly on different annular sets, we obtain the existence of multiple solutions.
\item[(c):] For a given operator, we may construct multiple radial energies by choosing different radial potentials \( \mathcal{F} \). Thus, even in the variational case, the energy corresponds to a particular choice of \( \mathcal{F} \).
\end{itemize}
The structure of the paper is as follows. In Section~2 we present the main abstract result, introducing the functionals \(\mathcal{F}\) and \(\mathcal{E}\), together with the associated Nehari manifold. Section~3 is devoted to applications: first, we discuss a nonvariational \(p\)-Laplace problem illustrating our theoretical findings; second, we construct an example of \(\mathcal{F}\) and \(T\) such that, for each direction $v$,  the associated radial energy functional \(\mathcal{E}(v)\) has a unique maximum  but  two critical points, contrary to the common setting in the literature, to which, however, our result still applies.

We conclude this section with a series of basic well-known results from literature. The first result is Vitali's theorem (see, e.g., \cite[Lemma 9.1]{p semilinear}) used in Lemma \ref{bine definire nemitski}. 
\begin{theorem}
    Let $\Omega \subset \mathbb{R}^N$ be a bounded open set and let 
$(u_k)$ be a sequence of functions from $L^q(\Omega; \mathbb{R}^n)$ 
($1 \leq q< \infty$) with $u_k(x) \to u(x)$ as $k \to \infty$, for a.e.\ 
$x \in \Omega$. Then $u \in L^q(\Omega; \mathbb{R}^n)$ and 
$u_k \to u$ in $L^q(\Omega; \mathbb{R}^n)$ as $k \to \infty$, if and 
only if for each $\varepsilon > 0$ there exists a $\delta_\varepsilon > 0$ 
such that
\[
    \int_D |u_k|^q \, dx < \varepsilon
\]
for all $k$ and every $D \subset \Omega$ with $\mu(D) < \delta_\varepsilon$.
\end{theorem}
The next two results are essential in Theorems \ref{t1} and \ref{teorema alternative neliniara}, respectively. 
The first one is the version in cones due to Krasnoselskii
and Ladyzenskii\ \cite{kl} (see also \cite[p.139]{ag},  \cite{guo}), of the classical theorem of Birkhoff and Kellogg invariant-direction
theorem \cite{bk} (see also  \cite[Theorem 6.6]{ag}) regarding
the existence of a `nonlinear' eigenvalue and eigenvector for compact maps
in Banach spaces.
\begin{theorem}[Krasnoselskii and Ladyzenskii]
\label{tkl}Let $X$ be a real Banach space, $U\subset X$ be an open bounded
set with $0\in U$, $K\subset X$ a cone, and $T:K\cap \overline{U}\rightarrow
K$ a completely continuous operator. If
\begin{equation*}
\inf_{x\in K\cap \partial U}\left\vert T\left( x\right) \right\vert >0,
\end{equation*}%
then, there exist $\lambda _{0}>0$ and $x_{0}\in K\cap \partial U$ such that
\begin{equation*}
x_{0}=\lambda _{0}T\left( x_{0}\right) .
\end{equation*}
\end{theorem}
The second result is the well-known "nonlinear alternative" from compact mappings
 (see, e.g., \cite[Theorem 5.2]{ag}). \begin{theorem}[Nonlinear alternative]
\label{tna}
Let $X$ be a real Banach space, $K \subset X$ a convex set, 
and $U \subset K$ an open bounded set with $0 \in U$. 
Suppose that $T : \overline{U} \to K$ is a completely continuous operator. 
Then at least one of the following two alternatives holds:
\begin{enumerate}[(a)]
    \item $T$ has a fixed point in $\overline{U}$,
    \item there exist $x \in \partial U$ and $\lambda \in (0,1)$ such that
    \[
        x = \lambda T(x).
    \]
\end{enumerate}
\end{theorem}



\section{Main results}\label{main results}

\subsection{Existence, localization and multiplicity in conical annular sets}

\label{s1}

Let $\left( X,\lvert \cdot \rvert \right) $ be a Banach space, $K\subset X$
be a nondegenerate cone (also called a wedge), i.e., a closed, convex set such that $\lambda K\subset K$ for all
$\lambda \in \mathbb{R}_{+}$,  and  $%
K\setminus \{0\}\neq \emptyset $. Note that, in particular, one may take $K=X$. For any $\rho >0,$ denote
\begin{equation*}
K_{\rho }:=\left\{ v\in K:~\left\vert v\right\vert =\rho \right\} .
\end{equation*}%
Let
\begin{equation*}
T\colon K\rightarrow K\ \ \ \text{and\ \ \ }\mathcal{F}:K\times K\times K_1\rightarrow
\mathbb{R}
\end{equation*}%
be a completely continuous operator and a continuous functional,
respectively.

For each $v\in K_1,$ we associate to $\mathcal{F}$ and $T$ the \textit{"radial energy functional"} in direction $v,$
\begin{equation*}
\mathcal{E}(v):(0,\infty)\rightarrow \mathbb{R},\ \ \ \mathcal{E}%
(v)(t)=\int_{0}^{t}\mathcal{F}(T(sv),sv,v)ds.
\end{equation*}%
It is straightforward to verify that
\begin{equation*}
    \frac{d}{dt} \mathcal{E}(v)(t) = \mathcal{F}(T(tv), tv, v).
\end{equation*}
Let \( r, R \) be  two real numbers such that \( 0 \leq r < R \leq \infty \), and define
\begin{equation*}
    K_{rR} := \left\{ u \in K \, : \, r \leq |u| \leq R \right\}.
\end{equation*}
Clearly, if \( r = 0 \), then
\begin{equation*}
    K_{0R} = \left\{ u \in K \, : \, |u| \leq R \right\};
\end{equation*}
if \( R = \infty \), then
\begin{equation*}
    K_{r\infty} = \left\{ u \in K \, : \, |u| \geq r \right\};
\end{equation*}
and if \( 0 = r ,\,  R = \infty \), then
\begin{equation*}
    K_{0\infty} = K.
\end{equation*}
Our main assumption on the functional $\mathcal{E}$ is the following:
\begin{description}
\item[(h1)] There exist $r_{0}>0$ and $R_{0}<\infty $ with $r\leq
r_{0}<R_{0}\leq R$ such that, for every \( v \in K_1 \), the mapping \( \mathcal{E}(v) \) attains a unique maximum \( t_v \) on the interval \( [r, R] \), and this maximum satisfies \( r_0 < t_v < R_0 \).
\end{description}
In particular, one may set
\begin{equation*}
r_{0}=r\ \ \text{if }r>0\ \ \ \text{and\ \ \ }R_{0}=R\ \ \text{if }R<\infty .
\end{equation*}

\begin{lemma}
\label{l1}Under the assumption {\normalfont (h1)}, the mapping
\begin{equation*}
v\in K_{1}\mapsto t_{v}
\end{equation*}%
is continuous.
\end{lemma}

\begin{proof}
Let \( v_k \in K_1 \) with \( v_k \to v^0 \) as \( k \to \infty \).
Clearly, \( v^0 \in K_1 \), and by assumption (h1), the sequence \( \left( t_{v_k} \right) \) is bounded.
In order to prove that \( t_{v_k} \to t_{v^0} \), it suffices to show that any convergent subsequence of \( \left (t_{v_k}\right) \) has limit \( t_{v^0} \) (see, e.g., \cite[Lemma 1.1]{herve}).
Thus, let \( \left (t_{v_k}\right)\) be a subsequence converging to some \( t^0 \in [r_0, R_0] \).
Using (h1) we have that 
\begin{equation*}
\mathcal{E}\left( v_{k}\right) \left( t_{v_{k}}\right) \geq \mathcal{E}%
\left( v_{k}\right) \left( t\right) \ \ \ \text{for all\ }t\in \left[ r,R%
\right] ,
\end{equation*}%
that is %
\begin{equation*}
\int_{0}^{t_{v_{k}}}\mathcal{F}(T(sv_{k}),sv_{k},v_k)ds\geq \int_{0}^{t}\mathcal{%
F}(T(sv_{k}),sv_{k},v_k)ds\ \ \left( t\in \left[ r,R\right] \right) .
\end{equation*}%
Passing to the limit we obtain
\begin{equation*}
\int_{0}^{t^{0}}\mathcal{F}\left(T(sv^{0}),sv^{0}, v^0\right)ds\geq \int_{0}^{t}\mathcal{F}%
(T(sv^{0}),sv^{0}, v^0)ds\ \ \left( t\in \left[ r,R\right] \right) ,
\end{equation*}%
so $t^{0}$ is a maximum on $\left[ r,R\right] $ of $\mathcal{E}(v^{0}).$
Again from (h1), we have $t^{0}=t_{v^{0}}.$ Since the convergent subsequence $\left (t_{v_k}\right)$ was arbitrarily chosen, we  infer that the entire sequence $\left( t_{v_{k}}\right) $ converges to $%
t_{v^{0}},$ which proves our statement.
\end{proof}

In what follows we denote
\begin{align*}
&U :=\left\{ tv:\ v\in K_{1},\ 0\leq t<t_{v}\right\}
\end{align*}
and \begin{equation*}
U_{b} :=\left\{ t_{v}v:\ v\in K_{1}\right\} .
\end{equation*}
We call \( U_b \) the \textit{Nehari-type manifold} associated with the functional $\mathcal{F}$, operator \( T \) and the conical annular region \( K_{r,R} \).

\begin{remark}
\label{r0}Under assumption (h1), any point $u\in U_{b}$ satisfies%
\begin{equation*}
r\leq r_{0}<\left\vert u\right\vert <R_{0}\leq R.
\end{equation*}%
Thus, any fixed point of \( T \) that belongs to \( U_b \) lies in the conical annular set \( K_{rR} \). Moreover, assuming that \( 0 < r < R < \infty \), if \( u \in U_b \) is a fixed point of \( T \), then the associated energy in the direction \( v = \frac{u}{|u|} \), that is \( \mathcal{E} \left( \frac{u}{|u|} \right)(t) \), attains its maximum over the interval \( [r, R] \) at the unique point \( t = t_{\frac{u}{|u|}} \).
Therefore, looking for a fixed point of the operator \( T \) within the Nehari-type manifold \( U_b \), we not only obtain a fixed point lying in a conical annular region, but also one that maximizes the energy functional along its own direction.

\end{remark}

\begin{remark}
\label{r1}Under assumption (h1), one has%
\begin{equation*}
\mathcal{F}\left( T\left( u\right) ,u,\frac{u}{|u|}\right) =0\ \ \ \text{for all }u\in
U_{b}.
\end{equation*}%
Indeed, if $u\in U_{b},$ then $u=t_{v}v$ for $v=\frac{u}{|u|}\in K_{1}.$ Also, $t_{v}$
being the maximum of $\mathcal{E}\left( v\right) $ on $\left[ r,R\right] $
located in the open interval $\left( r,R\right) ,$ one has
\begin{equation*}
\frac{d}{dt}\mathcal{E}\left( v\right) \left( t_{v}\right) =\mathcal{F}%
(T(t_{v}v),t_{v}v, v)=0.
\end{equation*}
\end{remark}

\begin{lemma}
\label{l2}Under the assumption {\normalfont (h1)}, the set $U$ is an open bounded subset of $%
K, $ $0\in U,$ and its boundary $\partial U$ relative to $K$ is $U_{b}.$
\end{lemma}

\begin{proof}
   Clearly, by definition, we have \( 0 \in U \).

   \textit{(i) Boundedness.} Let \( u \in U \). Then \( u = t v \), where \( v = \frac{u}{|u|} \in K_1 \) and \( t < t_v \). Thus, by assumption~(h1), we have
\begin{equation*}
    |u| \leq t < t_v < R_0,
\end{equation*}
and consequently,
\[
U \subset K_{R_0},
\]
which shows that \( U \) is bounded.

\textit{(ii) Openness.} To prove that $U$ is open we show that its
complement
\begin{equation*}
U_{c}:=\left\{ tv:\ v\in K_{1},\ t\geq t_{v}\right\}
\end{equation*}%
is closed. Let \( u_k \in U_c \) be a sequence such that
\begin{equation*}
    u_k \to u \in K.
\end{equation*}
We  show that \( u \in U_c \). By the definition of \( U_c \), we have\begin{equation*}
    u_k = t_k v_k,
\end{equation*} where \( t_k \geq t_{v_k} \) and \( |v_k| = 1 \).   From  the convergence $t_{k}v_{k}\rightarrow u$, it follows that
that $t_{k}\rightarrow \left\vert u\right\vert .$ Since \begin{equation}\label{ineq t_k}
    t_{k}\geq
t_{v_{k}}\geq r_{0}>0,
\end{equation}one has $\left\vert u\right\vert \geq r_{0},$ so $%
u\in K\setminus \left\{ 0\right\} .$ 

Next, observe that $v_{k}\rightarrow \frac{1}{%
\left\vert u\right\vert }u=:v^{0}$ and based on Lemma \ref{l1}, it follows that $%
t_{v_{k}}\rightarrow t_{v^{0}}.$ Passing to the limit in \eqref{ineq t_k}, we deduce that $\left\vert u\right\vert \geq t_{v^{0}}.$ Thus
\begin{equation*}
u=\left\vert u\right\vert v^{0}\geq t_{v^{0}}v^{0}
\end{equation*}%
where $v^{0}\in K_{1}.$ This shows that $u\in U_{c}$ as desired, and  therefore $U_{c}$
is closed.

\textit{(iii) $\partial U=U_b$.}
In order to prove that $\partial U=U_{b},$ we need to show that
\begin{equation}
\overline{U}=D:=\left\{ tv:\ 0\leq t\leq t_{v},\ v\in K_{1}\right\} .
\label{rp2}
\end{equation}%
Let $u\in \overline{U}$ be any point. Then $u$ is the limit of a sequence  $u_k\in U$, where 
 $u_k=t_{k}v_{k}$ with
\begin{equation*}
v_{k}\in K_{1},\ \ 0\leq t_{k}<t_{v_{k}}\ \ \text{and\ \ }%
t_{k}v_{k}\rightarrow u.
\end{equation*}%
Taking to a subsequence we may assume that $t_{k}\rightarrow t^{0}\in \left[
0,R_{0}\right] .$ If $t^{0}=0,$ then $u=0\in D.$ Otherwise, if $t^{0}>0,$ we
deduce that $v_{k}\rightarrow \frac{1}{t^{0}}u=:v^{0}$ and $v^{0}\in K_{1}.$
Also, based on Lemma \ref{l1}, one has $t_{v_{k}}\rightarrow t_{v^{0}},$ and
from $t_{k}<t_{v_{k}}$ we deduce that $t^{0}\leq t_{v^{0}}.$ Hence $%
u=t^{0}v^{0},$ where $0\leq t^{0}\leq t_{v^{0}}$ and $v^{0}\in K_{1},$ that
is $u\in D.$ Thus $\overline{U}\subset D.$

For the converse inclusion, take any $u\in D.$ Then
\begin{equation*}
u=tv\ \ \ \text{ for some\ \ }v\in K_{1}\ \ \text{and\ }t\ \text{with\ \ }%
0\leq t\leq t_{v}.
\end{equation*}%
If $t=0,$ then $u=0\in \overline{U}.$ For $t>0,$ we may choose an increasing
sequence $\left( t_{k}\right) $ with $0<t_{k}<t\leq t_{v}$ $\ $and $%
t_{k}\rightarrow t.$ Clearly $u_{k}:=t_{k}v\in U$ and $u_{k}\rightarrow
tv=u. $ Hence $u\in \overline{U}.$ Therefore $D\subset \overline{U}$
finishing the proof of (\ref{rp2}).

Finally, from the representations of $\overline{U}$ and $U,$ we see that $%
\partial U=\overline{U}\setminus U=U_{b}.$
\end{proof}

We now proceed to state the main result of the paper, a theorem of existence
and localization of a fixed point in a conical annular region. Consider the
following set of conditions:

\begin{description}
\item[(h2)] One has
\begin{equation*}
\inf_{u\in U_{b}}\left\vert T(u)\right\vert >0.
\end{equation*}

\item[(h3)] The functional $\mathcal{F}$ is such that if
\begin{equation*}
\mathcal{F}\left(t u,u,\frac{u}{|u|}\right)=0\quad \text{for some }t >0\text{ and }u\in K\setminus \{0\},
\end{equation*}%
then $t=1$.

\item[(h4)] The functional $\mathcal{F}$ is such that if
\begin{equation*}
\mathcal{F}\left(t u,u,\frac{u}{|u|}\right)=0\quad \text{for some }t >0\text{ and }u\in K\setminus \{0\},
\end{equation*}%
then $t\geq 1$.

\item[(h5)] The functional $\mathcal{F}$ is such that if
\begin{equation*}
\mathcal{F}\left(t u,u,\frac{u}{|u|}\right)=0\quad \text{for some }t >0\text{ and }u\in K\setminus \{0\},
\end{equation*}%
then $t\leq 1$.

\item[(h6)] The operator $T$ is compressive on $U$, i.e.,
\begin{equation*}
\left\vert Tu\right\vert \leq \left\vert u\right\vert \ \ \ \text{for all }%
u\in U_{b}.
\end{equation*}

\item[(h7)] The operator $T$ is expansive on $U$, i.e.,
\begin{equation*}
\left\vert Tu\right\vert \geq \left\vert u\right\vert \ \ \ \text{for all }%
u\in U_{b}.
\end{equation*}
\end{description}

\begin{theorem}
\label{t1}Let condition {\normalfont (h1)} holds. If either the set of conditions  {\normalfont (h2)} and
 {\normalfont (h3)}; or the set  {\normalfont (h2)},  {\normalfont (h4)} and  {\normalfont (h6)}; or the set  {\normalfont (h2)},  {\normalfont (h5)} and  {\normalfont (h7)} is
satisfied, then $T$ has a fixed point $u$ in $U_{b}.$ In addition, $u$
maximizes the radial energy functional $\mathcal{E}(v)$ on the interval $\left(
r,R\right) ,$ along its own direction $v=\frac{1}{\left\vert u\right\vert }%
u. $
\end{theorem}

\begin{proof}
(a) Let conditions (h2) and (h3) hold. Hypothesis (h2) makes applicable the
Birkhoff-Kellogg theorem for the operator $T$ on the set $U.$ Hence, it is
guaranteed the existence of a number $t >0$ and $u\in \partial U=U_{b}$ (from
Lemma \ref{l2}) such that $T(u)=t u.$ Since $u\in U_{b}$, using Remarks \ref%
{r1}, we have
\begin{equation}
0=\mathcal{F}\left(T(u),u,\frac{u}{|u|}\right)=\mathcal{F}\left(t u ,u,\frac{u}{|u|}\right),  \label{rp1}
\end{equation}%
so by assumption (h3), $t=1$, i.e., $u$ is a fixed point for $T$.

(b) Let conditions (h2), (h4) and (h6) hold. 
As above, there exist \( t > 0 \) and \( u \in U_b \) such that \( T(u) = t u \), and relation~\eqref{rp1} holds. Then, by assumption~(h4), we have \( t \geq 1 \). If \( t > 1 \), then
\[
|T(u)| = t |u| > |u|,
\]
which contradicts assumption~(h6). Hence, \( t = 1 \), and therefore \( u \) is a fixed point of \( T \).

(c) Let conditions (h2), (h5) and (h7) hold. As in step~(b), there exist \( t > 0 \) and \( u \in U_b \) such that \( T(u) = t u \), and relation~\eqref{rp1} holds. Now, from assumption~(h5), we have \( t \leq 1 \). If \( t < 1 \), then
\[
|T(u)| = t |u| < |u|,
\]
which contradicts assumption~(h7). Hence, \( t = 1 \), and thus \( T(u) = u \).
\end{proof}

Using the nonlinear alternative instead of the Birkhoff-Kellogg theorem, we obtain the following existence result.
\begin{theorem}\label{teorema alternative neliniara}
    Assume conditions {\normalfont (h1)} and {\normalfont (h5)} hold true. Then the operator $T$ has a fixed point in $u\in \overline{U}.$
\end{theorem}
\begin{proof}
    From the nonlinear
alternative we have that either
\begin{description}
\item[(a)] $T$ has a fixed point in $\overline{U},$
\end{description}

or

\begin{description}
\item[(b)] $T(u)=tu$ for some $u\in U_{b}$ and $t>1$.
\end{description}
The second case of the alternative does not occur, since otherwise we would
have
\begin{equation*}
0=\mathcal{F}\left( T\left( u\right) ,u,\frac{u}{|u|}\right)=\mathcal{F}\left( tu ,u,\frac{u}{|u|}\right),
\end{equation*}%
whence by (h5), we obtain $t\leq 1,$ which is a contradiction. 
\end{proof}
\begin{remark}
    Clearly, the
fixed point $u$ given by Theorem \ref{teorema alternative neliniara} can be the origin. However, if $T\left( 0\right) \neq 0,$
then $u$ is not trivial.
\end{remark}

\begin{remark}
\label{r3}All the previous results remain valid if, in condition~(h1), the maximum is replaced by the minimum, that is, the following condition holds:
\begin{description}
\item[(h1$^\ast$)] There exist $r_{0}>0$ and $R_{0}<\infty $ with $r\leq
r_{0}<R_{0}\leq R$ such that for each $v\in K_{1}$, the mapping $\mathcal{E}%
(v)$ has a unique minimum $t_{v}$ on $\left[ r,R\right] $ and $%
r_{0}<t_{v}<R_{0}.$
\end{description}
Indeed, (h1$^\ast$) implies (h1) for $-\mathcal{F},$ while all the other
conditions (h2)-(h7) are not affected by this change of sign.
\end{remark}

Noting that the definition of the Nehari-type manifold $U_{b}$ and all the
associated conditions are given with respect to a fixed pair $(r,R),$ we can
expect to repeat them for different such pairs, thereby obtaining a finite
or infinite number of fixed points of the operator $T$ in the cone $K.$ This
can happen if the mapping $\ \mathcal{E}$ has oscillations. Thus, we
have
\begin{theorem}
\label{t1 copy(1)}

\begin{description}
\item[(1$^{0}$)] If there are pairs of numbers $\left( r_{i},R_{i}\right) ,$
$i=1,2,...,n\ $ such that
\begin{equation*}
r_{i}<R_{i}\leq r_{i+1}\ \ {\ \text for\ } i=1,2,...,n-1
\end{equation*}
and Theorem \ref{t1} applies to each of these pairs, then the operator $T$
has $n $ fixed points $u_{i}\in K$ with%
\begin{equation*}
r_{i}<\left\vert u_{i}\right\vert <R_{i},\ \ \ i=1,2,...,n.
\end{equation*}

\item[(2$^{0}$)] If there is a sequence of pairs $\left( r_{i},R_{i}\right)
,\ i=1,2,...\ $ such that
\begin{equation*}
r_{i}<R_{i}\leq r_{i+1}\ \ {\ \text for\ } i=1,2,...
\end{equation*}
and Theorem \ref{t1} applies to each of these pairs, then the operator $T$
has a sequence of fixed points $u_{i}\in K$ with%
\begin{equation*}
r_{i}<\left\vert u_{i}\right\vert <R_{i},\ \ \ i=1,2,....
\end{equation*}

\item[(3$^{0}$)] If there is a sequence of pairs $\left( r_{i},R_{i}\right)
,\ i=1,2,...\ $ such that
\begin{equation*}
r_{i+1}<R_{i+1}\leq r_{i}\ \ {\ \text for\ } i=1,2,...
\end{equation*}
and Theorem \ref{t1} applies to each of these pairs, then the operator $T$
has a sequence of fixed points $u_{i}\in K$ with%
\begin{equation*}
r_{i}<\left\vert u_{i}\right\vert <R_{i},\ \ \ i=1,2,....
\end{equation*}
\end{description}
\end{theorem}

\subsection{Recover the classical variational framework}

\noindent
In Section~\ref{s1}, a Nehari-type method has been introduced in a nonvariational framework, starting from a fixed point problem
\begin{equation*}
T\left( u\right) =u
\end{equation*}%
associated with an operator $T.$ In this respect, an energy-type function and a Nehari-type manifold have been defined. The natural question is whether we can recover the classical Nehari method within the variational framework.

To give an answer, let $X^{\ast }$ be the dual of $X$, $\langle \cdot ,\cdot
\rangle $ denote the duality between $X^{\ast }$ and $X,$ and let the norms
on $X$ and $X^{\ast }$ be denoted by the same symbol $\left\vert \cdot
\right\vert .$ We shall denote by $J$ the \textit{duality mapping }%
corresponding to a normalization function $\theta ,$ i.e. the set-valued
operator $J:X\rightarrow {\mathcal{P}}(X^{\ast })$ defined by
\begin{equation*}
Jx=\big\{x^{\ast }\in X^{\ast }:\ \langle x^{\ast },x\rangle =\theta \left(
\vert x\vert \right) \left\vert x\right\vert ,\ \vert
x^{\ast }\vert =\theta \left( \vert x\vert \right)
\big\},\quad x\in X.
\end{equation*}%
Recall that by a normalization function we mean a continuous strictly
increasing function $\theta :%
\mathbb{R}
_{+}\rightarrow
\mathbb{R}
_{+}$ with $\theta \left( 0\right) =0$ and lim$_{t\rightarrow \infty
}\theta \left( t\right) =\infty .$ We assume that $\theta \left( 1\right)
=1.$ Obviously, one has%
\begin{equation*}
J\left( \lambda x\right) =\theta \left( \lambda \right) J(x)
\end{equation*}%
for every $x\in X$ and $\lambda \in
\mathbb{R}
.$ We assume that $J$ is single-valued, bijective, and that both $J$ and its
inverse $J^{-1}$ are continuous.

Furthermore, assume that \( E \in C^1(X) \) is a functional satisfying, without loss of generality, \( E(0) = 0 \). We can easily see that for any $t>0$ and $u\in X$, one has\begin{equation*}
    E(tu)=E(tu)-E(0)=\int_0^1 \langle E'( \theta tu), \,tu\rangle d\theta =\int_0^t \langle E'( s u), \,u\rangle ds.
\end{equation*}
Our  goal is to obtain critical point of $E,$ that is to solve the equation
\begin{equation*}
E^{\prime }\left( u\right) =0.
\end{equation*}%
Take $K=X$. We consider the operator $T:K\rightarrow K$ defined by%
\begin{equation*}
T\left( u\right) :=J^{-1}\left( Ju-E^{\prime }\left( u\right) \right).
\end{equation*}%
Obviously, the critical points of the functional $E$ are the fixed points of
the operator $T.$ 

Let us consider   the   functional $\mathcal{F}:K\times K\times
K_1\rightarrow
\mathbb{R}
$ (recall that $K_1=\{u\in K\, : \, |x|=1\}$),
\begin{equation*}
\mathcal{F}\left( u,v,w\right) :=\left\langle Jv-Ju,w\right\rangle.
\end{equation*}%
Then, 
\begin{align*}
\mathcal{E}\left( v\right) \left( t\right) &:=\int_{0}^{t}\mathcal{F}\left(
T\left( sv\right) ,sv,v\right) ds\\&=\int_{0}^{t}\left\langle J\left( sv\right)
-JT\left( sv\right) ,v\right\rangle ds\\&=\int_{0}^{t}\left\langle E^{\prime
}\left( sv\right) ,v\right\rangle ds\\& =E(tv).
\end{align*}%
From this, we observe that the classical energy functional $E$ is recovered through $\mathcal{E}$. Consequently, the classical Nehari manifold method (see \cite{nehari}) is also retrieved; both in terms of the maximized mapping 
$$
t \mapsto E(tu),
$$
and in the   classical Nehari manifold $\mathcal{N}$ given by
$$
\mathcal{N} = \left\{ u \in X : \langle E'(u), u \rangle = 0 \right\}.
$$

\begin{corollary}\label{corolar}
Under the previous conditions, if in addition
\begin{equation*}
J^{-1}\left( Ju-E^{\prime }\left( u\right) \right) (K)\subset K,
\end{equation*}%
then $E$ has a critical point $u$ in the Nehari manifold $U_{b},$ with%
\begin{equation*}
u\in K,\ \ r<\left\vert u\right\vert <R.
\end{equation*}%
In addition, $u$ maximizes the radial energy functional
\begin{equation*}
t\mapsto \mathcal{E}(v)\left( t\right) =E(tv),
\end{equation*}%
on the interval $\left( r,R\right) ,$ along its own direction $v=\frac{1}{%
\left\vert u\right\vert }u.$
\end{corollary}

\begin{proof}
We only need to check conditions (h2) and (h3).

\textit{Check of (h2)}: Let $u\in U_{b}$ and denote
\begin{equation*}
w:=T\left( u\right) =J^{-1}\left( Ju-E^{\prime }\left( u\right) \right) .
\end{equation*}
Then $\left\vert u\right\vert \geq r_{0}$ and
\begin{equation*}
Jw=Ju-E^{\prime }\left( u\right),
\end{equation*}%
whence, since $\left\langle E^{\prime }\left( u\right) ,u\right\rangle =0,$
one has
\begin{equation*}
\left\langle Jw,u\right\rangle =\left\langle Ju-E'(u),u\right\rangle=\left\langle Ju,u\right\rangle =\theta
\left( \left\vert u\right\vert \right) \left\vert u\right\vert \leq
\left\vert Jw\right\vert \left\vert u\right\vert =\theta \left( \left\vert
w\right\vert \right) \left\vert u\right\vert .
\end{equation*}%
From this, we have \( \theta(|w|) \geq \theta(|u|) \). Thus, by the monotonicity of the function \( \theta \), it follows that
\[
|w| \geq |u| \geq r_0.
\]
Therefore,
\begin{equation*}
\inf_{u\in U_{b}}\left\vert T\left( u\right) \right\vert \geq r_{0}>0.
\end{equation*}

\textit{Check of (h3)}: Assume that 
\begin{equation*}
    \mathcal{F}\left(t u,u,\frac{u}{|u|}\right)=0,
\end{equation*}
for some $t>0$ and $%
u\in K\setminus \{0\}.$ Then%
\begin{equation*}
0=\left\langle Ju-J\left( tu\right) ,\frac{u}{|u|}\right\rangle,
\end{equation*}%
which is equivalent to \begin{equation*}
   0 =\left\langle Ju-\varphi
\left( t\right) Ju,u\right\rangle =\left( 1-\varphi \left( t\right) \right)
\varphi \left( \left\vert u\right\vert \right) \left\vert u\right\vert .
\end{equation*}
Hence, \( \varphi(t) = 1 \), and therefore \( t = 1 \), according to our assumption that \( \varphi(1) = 1 \).

Thus Theorem \ref{t1} applies and gives the result.
\end{proof}

\section{Applications}
\subsection{A nonvariational Dirichlet problem with $p$-Laplacian}
To illustrate the theoretical results, we first consider the following nonvariational  Dirichlet problem for
a $p$-Laplace equation
\begin{equation}
\begin{cases}
-\left( \left\vert u^{\prime }\right\vert ^{p-2}u^{\prime }\right) ^{\prime
}(t)=f(u(t), u'(t)),\quad t\in \left( 0,1\right),  \\
u(t)\geq 0 \\
u(0)=u(1)=0,%
\end{cases}
\label{pa}
\end{equation}%
where $1<p<\infty $, and $f:\mathbb{R}_+\times \mathbb{R}\rightarrow \mathbb{R}_+$ is a 
continuous function.

Consider the Banach space $X:=W_{0}^{1,p}\left( 0,1\right) $ endowed with
the usual norm $|u|_{1,p}:=|\nabla u|_{p}$, where \begin{equation*}
    |w|_p=\left(\int_0^1 |w|^pdx\right)^{\frac{1}{p}}\quad (w\in L^p(0,1)).
\end{equation*}It is well known (see, e.g.,
\cite[Chapter~1.2]{djm}) that $W_{0}^{1,p}(0,1)$ is a uniformly convex and
reflexive Banach space with $W_{0}^{-1,q}(0,1)$ its dual, where $%
q$ is the conjugate of $p$, i.e., $\frac{1}{p}+\frac{1}{q%
}=1$. Let $\langle \cdot ,\cdot \rangle $ denote the duality pairing between
$W_{0}^{-1,q}(0,1)$ and $W_{0}^{1,p}(0,1)$.  Then, for any $v\in
L^{q}(0,1)\subset W_{0}^{-1,q}(0,1)$, we have
\begin{equation*}
\langle v,u\rangle =\int_{0}^{1}v(t)u(t)dt,
\end{equation*}%
for all $u\in W_{0}^{1,p}(0,1)$ (see, e.g., \cite[Proposition~8.14]{brezis}).

In what follows, \( J \colon W^{1,p}_0(0,1) \to W^{-1,q}_0(0,1) \) denotes the duality mapping of \( W^{1,p}_0(0,1) \) corresponding to the normalization function \( \theta(t)=t^{p-1} \) (\( t \geq 0 \)).  
It is known that
\[
    Ju = -\left( |u'|^{p-2} u' \right)', \quad u \in W^{1,p}_0(0,1),
\]
and that \( J \) is bijective and continuous, while its inverse \( J^{-1} \) is strongly monotone, bounded (maps bounded sets into bounded sets), and continuous (see, e.g., \cite[Theorem~8]{djm}).

Let $\lambda _{p}$ denote the first eigenvalue of the Euler-Lagrange
equation 
\begin{equation*}
Ju =\lambda |u|_{1,p}^{p-2}u\quad \text{in }\,(0,1),\quad
u(0)=u(1)=0.
\end{equation*}%
Then (see, e.g., \cite[Remark~6]{djm}),
\[
\lambda_{p} = \min_{u \in W_{0}^{1,p}(0,1) \setminus \{0\}} \frac{|u|_{1,p}^{p}}{|u|_{p}^{p}},
\]
that is, \( c_{p} := \lambda_{p}^{-1/p} \) is the smallest constant such that, for all \( u \in W_{0}^{1,p}(0,1) \), we have
\begin{equation}\label{sobolev}
    |u|_{p} \leq c_{p} \, |u|_{1,p}.
\end{equation}
 Given the continuous
embedding of  $W_{0}^{1,p}(0,1)$   in $C[0,1]$ (see e.g.,    \cite{brezis}), one has
\begin{equation}\label{continuitate}
\left\vert u(t)\right\vert \leq |u|_{1,p}\ \ \ \left( t\in \left[ 0,1\right]
\right) ,
\end{equation}%
for every $u\in W_{0}^{1,p}(0,1)$.

In $W^{1,p}_0(0,1)$, we consider the cone
\begin{equation*}
\begin{aligned} K = \Big\{ u \in W^{1,p}_0(0,1) \, : \, & u \geq 0, \,u \text{
is concave}, \, u(t) = u(1-t)\\ &\text{ and } u(t) \geq
\phi(t) \, |u|_{1,p} \text{ for all } t \in [0,1/2] \Big\}, \end{aligned}
\end{equation*}%
where $\phi $ is given by
\begin{equation}\label{phi}
\phi :[0,1/2]\rightarrow \mathbb{R}_+,\text{ \quad $\phi
(t)=\int_0^t (1-2s)^{\frac{1}{p-1}} ds$.}
\end{equation}
For $0<r<R<\infty $, we assume that the following conditions hold:
\begin{description}
\item[(H1)]
For each \( x \in \mathbb{R}_+ \), the function \( f(x,\cdot) \) is decreasing on $\mathbb{R}_+$, and for each \( y \in \mathbb{R} \), the function \( f(\cdot,y) \) is increasing on $\mathbb{R}_+$. Moreover, \( f \) is even in its second variable, that is,
\[
    f(x, y) = f(x, -y),
\]
for all \( x\in \mathbb{R}_+\) and \( y \in \mathbb{R} \).

\item[(H2)] 
Let \begin{equation*}
    f_0(x)=f(x,0)\quad \text{ and }\quad f_\infty(x)=\lim_{y\to \infty}f(x,y).
\end{equation*}
The functions \( f_0 \) and \( f_\infty \) satisfy 
\begin{equation}\label{conditie r}
f_0(r)<\frac{r^{p-1}}{c_{p}}\quad \text{ and }\quad f_\infty(R\phi
(\beta ))>\frac{R^{p-1}}{2\Phi
},
\end{equation}%
where $\beta\in \left(0,\frac{1}{2} \right)$ is a fixed value, and 
\begin{equation*}
\Phi :=\int_{\beta }^{1/2}\phi (s)ds.
\end{equation*}%

\item[(H3)] For each $x\in (0,\infty)$ and $y\in \mathbb{R}$, the mapping
\begin{equation*}
t\mapsto \frac{f(tx,ty)}{t^{p-1}}\,\text{ }
\end{equation*}%
is strictly increasing on $(0,R].$
\end{description}
In order to apply Theorem~\ref{t1}, we choose the  functional \( \mathcal{F}\colon K\times K\times K_1 \to \mathbb{R} \) to be
 \begin{equation*}
    \mathcal{F}(u,v,w)=\langle J v-Ju,w\rangle.
\end{equation*}
Clearly $\mathcal{F}$ is continuous as $J$  is continuous.

We observe that the problem \eqref{pa} allows for a fixed point formulation \begin{equation*}
    u=T(u),
\end{equation*}
where $T\colon K \to W^{1,p}_0(0,1)$ is given by \begin{equation}\label{T}
    T(u)=J^{-1} N_f(u,u'),
\end{equation}
and \begin{equation}\label{N_f}
   N_f \colon K \times L^p(0,1) \to L^{q}(0,1),\quad N_f(u,\Tilde{u})(t)=f(u(t),\Tilde{u}(t)),
\end{equation}is the Nemytskii operator associated with the function $f$.


In the subsequent, we will present a series of auxiliary results that will be used to prove the invariance of the operator $T$ over $K$. When we say that a function \( u \) is symmetric, we understand symmetric with respect to \( \frac{1}{2} \).

The first result concerns the symmetry and nonnegativity of the solution to a \( p \)-Laplace equation.
\begin{lemma}\label{lema positiv si simetric}
    Let \( h \in L^{q}(0,1) \) be nonnegative a.e. on $(0,1)$ and symmetric. Then so is \( J^{-1}h \).
\end{lemma}
\begin{proof}
   Clearly $J^{-1}h$ is well defined (recall that $L^{q}(0,1)\subset W^{-1,q }_0$(0,1)).  Nonnegativity follows directly from the comparison principle for the \( p \)-Laplace operator (see \cite[Lemma~1.3]{ac}).  Let \( u_h = J^{-1}h \) and define \( w(t) = u_h(1 - t) \).  
A simple computation shows that \( Jw = h \).  
Thus, by the uniqueness of the solution to the \( p \)-Laplace equation, we deduce that \( u_h = w \), which proves that \( J^{-1}h \) is symmetric.
\end{proof}

Our second result concerns the preservation of concavity for a \( p \)-Laplace equation.  
Before stating it, let us define
\[
    \varphi \colon \mathbb{R} \to \mathbb{R}, \quad
    \varphi(x) = |x|^{p-2} x = \operatorname{sgn}(x)\, |x|^{p-1}.
\]
Clearly, \( \varphi \) is a homeomorphism from \( \mathbb{R} \) onto \( \mathbb{R} \), and its inverse \( \varphi^{-1} \) is increasing on \( \mathbb{R} \).  
Moreover, the problem \eqref{pa} can be written in terms of \( \varphi \) as
\[
    -\left( \varphi(u') \right)' = f(u, u'), \quad u(0) = u(1) = 0.
\]

\begin{lemma}\label{lema reprezentare}
    For any $h\in L^{q}(0,1)$ that is symmetric with respect to $\frac{1}{2}$,  we have the representation
    \begin{equation}\label{reprezentare J la -1}
    \left(J^{-1} h\right)(t)=\int_0^t (\varphi^{-1}\circ \mu)(s)ds,
\end{equation}
where\begin{equation*}
    \mu(t)=\int_t^{\frac{1}{2}} h(s)ds.
\end{equation*}
If in addition $h$ is
    nonnegative a.e. on  $(0,1)$ and nondecreasing a.e. on $\left(0,\frac{1}{2}\right)$, then   $ J^{-1}h$ is concave on $(0,1)$.
\end{lemma}
\begin{proof}
Denote
\[
w(t) = \int_0^t (\varphi^{-1} \circ \mu)(s) \, ds,
\]
Since \( \mu \) is continuous (see, e.g., \cite[Theorem~8.2]{brezis}), it follows that \( w \in C^1[0,1] \). Thus,
\[
\varphi(w') = \mu,
\]
and by \cite[Lemma~8.2]{brezis}, we have that
\[
-\varphi(w')' = h.
\]
By the uniqueness of the solution to the \( p \)-Laplace equation (see, e.g., \cite[Remark~3]{djm}), we obtain \eqref{reprezentare J la -1} as the unique representation of  \( J^{-1}h \).

Assume now that \( h \) is nonnegative a.e.\ on \( (0,1) \) and nondecreasing a.e.\ on \( \left(0,\frac{1}{2}\right) \).
We observe that \( \mu' = -h \le 0 \) a.e.\ on \( (0,1) \); hence \( \mu \) is nonincreasing on \( (0,1) \) (see, e.g., \cite[Chapter~3]{folland}).
Since
\[
    (J^{-1} h)' = \varphi^{-1} \circ \mu,
\]
and the right-hand side is nonincreasing on \( (0,1) \) (as the composition of an increasing and a nonincreasing function), it follows that \( J^{-1} h \) is concave on \( (0,1) \).
\end{proof}
In the following, we establish a Harnack type inequality. 
\begin{lemma}\label{lema harnack}
 Let \( h \in L^{q}(0,1) \) be nonnegative a.e. on $(0,1)$,  symmetric with respect to \( \frac{1}{2} \), and nondecreasing a.e. on \( \left(0,\frac{1}{2}\right ) \). Then, for all $t\in $  \( \left(0,\frac{1}{2}\right ) \), one has  \begin{equation}\label{Harnack dem}
        \left(J^{-1}h\right)(t)\geq \phi(t) \left|  J^{-1}h\right|_{1,p},
    \end{equation}
    where $\phi$ is given in \eqref{phi}.
\end{lemma}
\begin{proof}
Let \( u_h = J^{-1}h \). 
Since \( u_h \) is symmetric with respect to \( \frac{1}{2} \) (see Lemma~\ref{lema positiv si simetric}), one has (recall that $u_h\in C^1 [0,1]$),
\[
u_h'\left( \frac{1}{2} \right) = 0.
\]
It is not difficult to prove that
\begin{equation}\label{ineq 1}
    u_h'(t) \geq (1 - 2t)^{\frac{1}{p-1}} u_h'(0),
\end{equation}
for all $t\in \left(0, \frac{1}{2} \right).$ To see this, let
\[
\sigma\colon \left(0, \frac{1}{2} \right)\to \mathbb{R}, \quad \sigma(t) = u_h'(t)^{p-1} - (1 - 2t) u_h'(0)^{p-1}.
\]
Since \( u_h \) is concave and \( u_h'\left( \frac{1}{2} \right) = 0 \), we deduce that \( u_h'(t) \geq 0 \) for all \( t \in \left(0, \frac{1}{2} \right) \).  
Thus,
\[
    u_h'(t)^{p-1} = -J u_h,
\]
which implies
\[
    \left( u_h'(t)^{p-1} \right)' = -h.
\]
Therefore,
\[
    \sigma'(t) = -h(t) + 2 u_h'(0)^{p-1}.
\]
which is nonincreasing a.e.\ on \( \left(0, \frac{1}{2} \right) \). Following \cite[Remark~2.5]{jp}, we conclude  that \( \sigma \) is concave. As \( \sigma(0) = \sigma\left( \frac{1}{2} \right) = 0 \), we have that \( \sigma(t) \geq 0 \) for all \( t \in \left(0, \frac{1}{2} \right) \), i.e., \eqref{ineq 1} holds.

Since $u_h(0)=0$, integrating  \eqref{ineq 1} from $0$ to $t$ ($t\leq \frac{1}{2}$), we obtain \begin{equation*}
    u_h(t)\geq u'_h(0) \int_0^t (1-2s)^{\frac{1}{p-1}}ds=\phi(t) u'_h(0).
\end{equation*}
Finally, the conclusion follows from the above inequality and
 \begin{equation*}
     |u_h|_{1,p } =\left(\int_0^1 |u'(s)|^pds\right)^{\frac{1}{p}}=\left(2\int_0^{1/2} u'(s)^pds \right)^{\frac{1}{p}}\leq \left(u'(0)^p\right)^{\frac{1}{p}}=u_h'(0).
 \end{equation*}
\end{proof}
Now we prove that the Nemytskii's operator $N_f$ is well-defined, bounded (maps bounded sets into bounded sets) and continuous.
\begin{lemma}\label{bine definire nemitski}
     Assume condition $\text{(H1)}$ holds true. Then the Nemytskii operator $N_f$ is well-defined, bounded (maps bounded sets into bounded sets) and continuous from $K \times L^p(0,1)$  to $L^{q}(0,1)$.
\end{lemma}

\begin{proof}
From (H1), we have
\begin{equation}\label{marginire f}
    0 \leq f(x,y) \leq f(x,0),
\end{equation}
for all \( x \in \mathbb{R}_+ \) and \( y \in \mathbb{R} \).  
Indeed, for \( x \in \mathbb{R}_+ \) and \( y \in \mathbb{R} \), the evenness of \( f \) in its second variable gives \( f(x,y) = f(x,|y|) \).  
Since \( f(x,\cdot) \) is decreasing on \( \mathbb{R}_+ \), we obtain \( f(x,|y|) \leq f(x,0) \).  
The left-hand side of \eqref{marginire f} follows from the definition of \( f \).

\textit{Well-definedness.} Let $u\in K$ and $\Tilde{u}\in L^p(0,1).$
   Since \( f \) is continuous, it follows that \( N_f(u,\Tilde{u}) \) is measurable (see, e.g., \cite[Proposition~9.1]{p semilinear}).  
Moreover, as \( f(u,0) \in C[0,1] \), using \eqref{marginire f} we obtain
\[
    \int_0^1 f(u,\tilde{u})^{q} \, dx 
    \leq \int_0^1 f(u,0)^{q} \, dx 
    = |f(u,0)|_{q}^{q},
\]
which shows that \(N_f(u,\Tilde{u}) \in L^{q}(0,1) \).

\textit{Boundedness.} Let $u\in K$ and $\Tilde{u}\in L^p(0,1).$
Using \eqref{marginire f}, to prove the boundedness of \( N_f \), it suffices to show that the operator \( u \mapsto f(u,0) \) maps bounded sets into bounded sets from \( W^{1,p}_0(0,1) \) to \( L^{q}(0,1) \). This however  follows immediately from the continuous embedding of \( W^{1,p}_0(0,1)\)  into $ C[0,1] \) and the continuity of \( f \).

\textit{Continuity.} Let $(u_k,\Tilde{u}_k)\in K\times L^p(0,1)$ be such that $(u_k,\Tilde{u}_k)\to (u,\Tilde{u})$ in $W^ {1,p}_0(0,1) \times L^p(0,1)$ as $k\to \infty$. Then, there exists a subsequence  (still denoted by $(u_k,\Tilde{u}_k)$) with the property that (see, e.g., \cite[Theorem 4.9]{brezis}) \begin{equation*}
    (u_k(x),\Tilde{u}_k(x))\to (u(x), \Tilde{u}(x)) \, \, \text{ a.e. on }(0,1).
\end{equation*}
    Passing again  to a subsequence if necessarily, we may assume that \begin{equation*}
        N_f(u_k,\Tilde{u}_k)(x)\to N_f(u,\Tilde{u})(x)\, \, \text{ a.e. on }(0,1).
    \end{equation*}
    Since $u_k\to u$ in $W^{1,p}_0(0,1)$, using \eqref{continuitate}, there exists $M>0$ such that $|u_k(t)|\leq M$, for all $t\in [0,1]$ and $k\in \mathbb{N}$.  Denote $M_1 := \sup_{x \in [0,M]} |f(x,0)|$.

Let \(\varepsilon > 0\) and set \(\delta = \frac{\varepsilon^{q}}{M_1}\). 
Then, for any \(D \subset [0,1]\) with \(meas(D) < \delta\), using \eqref{marginire f}, we estimate
\begin{equation*}
    |N_f(u_k,\Tilde{u}_k)|_{L^{q}}\leq \left(\int_D |f(u_k(x),0)|^{q}dx\right)^{\frac{1}{q}}\leq meas(D)^{\frac{1}{q}} M_1\leq \delta^{\frac{1}{q}}M_1=\varepsilon.
\end{equation*}
By Vitali's theorem, \(N_f(u_k,\tilde{u}_k)\) converges to \(N_f(u,\tilde{u})\) in \(L^{q}(0,1)\).  
Since the limit is the same regardless of the chosen subsequence, we conclude that the entire sequence \(N_f(u_k,\tilde{u}_k)\) converges to \(N_f(u,\tilde{u})\) in \(L^{q}(0,1)\) (see, e.g., \cite[Lemma~1.1]{herve}).
\end{proof}
Now we are ready to show that $T(K)\subset K$. \begin{lemma}\label{invarianta T K}
    Assume condition $\text{(H1)}$ holds true. Then, \begin{equation*}
        T(K)\subset K.
    \end{equation*}
\end{lemma}
\begin{proof}
Let \(u \in K\). Then, by Lemma~\ref{bine definire nemitski}, we have 
\(N_f(u,u') \in L^{q}(0,1)\). Since \(u \geq 0\), it follows that 
\(f(u,u') \geq 0\). The symmetry of \(u\) implies that \(u'\) is 
antisymmetric, i.e.,
\[
    u'(t) = -u'(1-t) \quad \text{a.e.\ on \((0,1)\)}.
\]
Thus, condition (H1) ensures that \(N_f(u,u')\) is symmetric. Moreover, since \(u\) 
is nondecreasing on \(\left(0,\frac{1}{2}\right)\) and \(u'\) is 
nonincreasing a.e.\ on \(\left(0,\frac{1}{2}\right)\), condition (H1) also 
ensures that \(N_f(u,u')\) is nondecreasing a.e.\ on 
\(\left(0,\frac{1}{2}\right)\). Now applying Lemmas \ref{lema positiv si simetric}, \ref{lema reprezentare} and \ref{lema harnack}, we find that $J^{-1}f(u,u')=T(u)$ is nonegative, symmetric, concave and satisfies the Harnack inequality \eqref{Harnack dem}. Therefore, $T(u)\in K$.
\end{proof}
Next, we prove the complete continuity of the operator $T$.
\begin{lemma}\label{T compact}
   Under assumption $\text{(H1)}$, the operator $T$ is completely continuous from $K$ to $K$.
\end{lemma}
\begin{proof}
    We observe that \begin{equation*}
        T=J^{-1}\circ P\circ N_f \circ I, 
    \end{equation*}
    where \begin{equation*}
        P\colon K\to K \times L^p(0,1), \quad P(u)=(u,u'),
    \end{equation*}
    and $I$ is the embedding operator \begin{equation*}
        I\colon L^{q}(0,1)\to W^{-1,q}_0(0,1),\quad I(u)=\langle u,\cdot\rangle.
    \end{equation*}
 Since, by the Rellich--Kondrachov theorem (see, e.g., \cite{brezis}), 
\(W^{1,p}_0(0,1)\) embeds compactly into \(L^{p}(0,1)\), it follows that 
\(L^{q}(0,1)\) embeds compactly into \(W^{-1,q}_0(0,1)\), as  
the dual spaces of \(L^{p}(0,1)\) and \(W^{1,p}_0(0,1)\), respectively. This follows from the Schauders's theorem on compact operators (see, \cite[Theorem 6.4]{brezis}, \cite[Theorem 4.19]{rudinFA} or \cite[Chapter 9]{ka}). Additionally, it is clear that \(I\) is  continuous and bounded.
Since \( J^{-1} \), \( P \), \( N_f \), and \( I \) are continuous and bounded, and one of these operators is compact, it follows that \( T \)  is completely continuous, as the composition of all these four operators.


\end{proof}

\begin{theorem}
\label{pa1} Under conditions $\text{(H1)-(H3)}$, the problem~\eqref{pa} admits a
solution $u\in K$ such that
\begin{equation*}
r<|u|_{1,p}<R.
\end{equation*}
Moreover, $u$ maximizes the energy function $\mathcal{E}(v)$ on the interval $\left(
r,R\right) ,$ along its own direction $v=\frac{1}{\left\vert u\right\vert }%
u. $
\end{theorem}
\begin{proof}
    \textit{Check of (h1). } We will show that for any $v\in K_1$, we have that \begin{equation}\label{capete}
        \frac{d}{dt}\mathcal{E}(v)(t)|_{t=r}>0,\quad
\frac{d}{dt}\mathcal{E}(v)(t)|_{t=R}<0,    \end{equation} 
and the mapping \begin{equation}\label{unica zero}
    t\mapsto  \frac{d}{dt}\mathcal{E}(v)(t),
\end{equation}
has a unique zero. It is easily to see that whenever \eqref{capete} and \eqref{unica zero} hold, then the function $t\mapsto \mathcal{E}(v)(t)$ has a unique maximum, as required by the abstract result. 

For any $v\in K_1$ and $t>0$, one has \begin{align*}
    \frac{d}{dt}\mathcal{E}(v)(t)&=t^{p-1} |v|_{1,p}^p-\int_0^1 f\left(t v(s),t v'(s)\right)v(s)ds\\&=t^{p-1}\left(|v|_{1,p}^p-\int_0^1 \frac{f\left(t v(s),t v'(s)\right)}{t^{p-1}}v(s)ds \right).
\end{align*}
Using \eqref{marginire f} and the monotonicity of \( f \) in the first variable, we obtain
\begin{equation}\label{estimare f r}
    f(rv(s), rv'(s)) \leq f(rv(s), 0) \leq f(r,0) = f_0(r), \quad s \in (0,1),
\end{equation}
where the latter inequality follows from \eqref{continuitate}, which ensures that \( v(s) \leq 1 \) for all \( s \in [0,1] \).


Thus, \eqref{estimare f r} yields
\begin{equation*}
    \int_0^1 \frac{f\left(r v(s),r v'(s)\right)}{r^{p-1}}v(t)ds\leq \frac{f_0\left(r\right)}{r^{p-1}}\int_0^1 v(s)ds<\frac{1}{c_p}|v|_{p}\leq 1.
\end{equation*}
Consequently, \begin{equation*}
    \frac{d}{dt}\mathcal{E}(v)(t)|_{t=r}=r^{p-1}\left(1-\int_0^1 \frac{f\left(r v(s),r v'(s)\right)}{r^{p-1}}v(s)ds \right)>0,
\end{equation*}
for all $v\in K_1$.

To obtain the second inequality in \eqref{capete}, 
observe that for any \( v \in K_1 \), the Harnack inequality, together with the monotonicity of \( v \) on \( \left(0,\frac{1}{2}\right) \), gives
\begin{equation}\label{v ineq}
        v(t) \geq \phi(t) \,\,\, \text{for all } t \in \left(0,\frac{1}{2}\right) \quad \text{ and }\quad 
    v(t) \geq \phi(\beta) \,\,\, \text{for all } t \in \left(\beta,\frac{1}{2}\right).
\end{equation}
From \eqref{v ineq} we obtain
\begin{equation}\label{Phi ineq}
    \int_{\beta}^1 v(s) \, ds \geq \Phi.
\end{equation}
Next, by (H1) we have
\begin{equation}\label{f inf ineq}
    f(x,y) \geq f_\infty(x), \quad x \in \mathbb{R}_+, \ y \in \mathbb{R}.
\end{equation}
Therefore, the monotonicity of \( f \) in its first variable, together with \eqref{v ineq} and \eqref{f inf ineq}, ensures that
\begin{equation}\label{f(Rphi)}
    f(Rv(s),Rv'(s)) \geq f_\infty(Rv(s)) 
    \geq f_\infty(Rv(\beta)) 
    \geq f_\infty(R\phi(\beta)), 
\end{equation}
for all $s\in \left(\beta,\frac{1}{2}\right)$.
Combining the second inequality from (H2), \eqref{Phi ineq} and \eqref{f(Rphi)}, we estimate
 \begin{align}
\nonumber    \int_0^1 \frac{f(R v(s), R v'(s))}{R^ {p-1}}v(s)ds
 &  =2\int_{0}^{\frac{1}{2}} \frac{f(R v(s), R v'(s))}{R^ {p-1}}v(s)ds\\&
\nonumber   \geq 2\frac{f(R\phi(\beta))}{R^{p-1}}\int_{\beta}^{\frac{1}{2}}v(s)ds\\&
\geq 2\frac{f(R\phi(\beta))}{R^{p-1}}\Phi.\nonumber\\&>1\nonumber.
\end{align}
Thus, \begin{equation*}
    \frac{d}{dt}\mathcal{E}(v)(t)|_{t=R}=R^{p-1}\left(1-\int_0^1 \frac{f\left(R v(s),R v'(s)\right)}{R^{p-1}}v(s)ds \right)<0,
\end{equation*}
for all $v\in K_1$, so
 both inequalities in \eqref{capete} are valid.  

To verify \eqref{unica zero}, we note that by (H3), for each $v\in K_1$, the mapping \begin{equation*}
    t\mapsto\int_0^1 \frac{f\left(t v(s),t v'(s)\right)}{t^{p-1}}v(s)ds
\end{equation*}
is strictly increasing. Therefore, by \eqref{capete}, we have that 
\begin{equation*}
        1-\int_0^1 \frac{f\left(r v(s),r v'(s)\right)}{r^{p-1}}v(s)ds>0 \quad\text{ and }\quad        1-\int_0^1 \frac{f\left(R v(s),R v'(s)\right)}{R^{p-1}}v(s)ds<0 ,
\end{equation*}
so the equation \begin{equation*}
    1-\int_0^1 \frac{f\left(t v(s),t v'(s)\right)}{t^{p-1}}v(s)ds=0,
\end{equation*}
has a unique solution. This however implies that, for all $v\in K_1,$ the mapping\begin{equation*}
     t\mapsto \frac{d}{dt}\mathcal{E}(v)(t)=t^{p-1}\left(1-\int_0^1 \frac{f\left(t v(s),t v'(s)\right)}{t^{p-1}}v(s)ds \right),
\end{equation*}
has a unique zero, as desired.

\textit{Check of (h2).} Let $u\in U_b$. Then, by definition, we have\begin{equation*}
0=\mathcal{F}\left(T(u), u,\frac{u}{|u|_{1,p}} \right)=\left \langle Ju-J(T(u)),\frac{u}{|u|_{1,p}}\right\rangle,
\end{equation*}
which gives, \begin{equation*}
    \langle J(T(u)), u\rangle=\langle J u, u\rangle.
\end{equation*}
One has, \begin{equation*}
    \theta(|u|_{1,p})|u|_{1,p}=\langle J u, u\rangle=\langle J(T(u)), u\rangle\leq |J(T(u))|_{-1,q}\, |u|_{1,p}=\theta(|T(u)|_{1,p})|u|_{1,p},
\end{equation*}
whence
 \begin{equation*}
    \theta(|u|_{1,p}) \leq \theta(|T(u)|_{1,p}).
\end{equation*} By the monotonicity of \(\theta\) we obtain
\[
    |T(u)|_{1,p} \geq |u|_{1,p} \geq r,
\]
which shows that (h2) is satisfied.

\textit{Check of (h3).} Since $\theta(t)=t^{p-1}$, as shown in Corollary \ref{corolar}, the functional $\mathcal{F}$ satisfies (h3).

The conclusion follows from Theorem~\ref{t1}, as \( T \) is completely continuous and invariant over \( K \) (see Lemmas~\ref{invarianta T K} and \ref{T compact}), and conditions (h1)–(h3) are satisfied.
\end{proof}
\begin{remark}
    If the inequalities in \eqref{capete} are reversed, i.e.,
\begin{equation}\label{capete2}
    \left.\frac{d}{dt} \mathcal{E}(v)(t)\right|_{t=r} < 0, \quad
    \left.\frac{d}{dt} \mathcal{E}(v)(t)\right|_{t=R} > 0,
\end{equation}
then Theorem~\ref{t1} is still applicable.  
Indeed, in this case, the point \( t_v \) is a minimum point of the mapping \( t \mapsto \mathcal{E}(v)(t) \), so condition (h1$^{\ast}$) is satisfied instead of (h1).  
As noted in Remark~\ref{r3}, the theory still applies and yields the existence of a point \( u \in K \) such that \( r < |u|_{1,p} < R \), which, in this case, minimizes the associated energy functional.

It is straightforward to verify that if
\begin{equation}
    \frac{f_\infty(r\phi(\beta))}{(r\phi(\beta))^{p-1}} > \frac{1}{2\Phi \, \phi(\beta)^{p-1}}
    \quad \text{and} \quad
    \frac{f_0(R)}{R^{p-1}} < \frac{1}{c_{p}},
\end{equation}
are assumed instead of \eqref{capete}, then \eqref{capete2} holds.

\end{remark}
\begin{example}
A typical example of a function \( f \) that satisfies conditions (H1)–(H3) is
\[
    f(x,y) = x^{p} \left( 1 + \frac{1}{1 + |y|} \right), \quad x\in \mathbb{R}_+,\, y\in \mathbb{R}.
\]
Clearly, \( f \) satisfies (H1). Moreover, \( f_0(x) = |x|^{p} \) and \( f_\infty(x) = 2|x|^{p} \), and they satisfy
\[
    \lim_{x \to 0} \frac{f_0(x)}{x^{p-1}} 
    =  \lim_{x \to 0} \frac{f_\infty(x)}{x^{p-1}} = 0,
\]
and
\[
    \lim_{x \to \infty} \frac{f_0(x)}{x^{p-1}} 
    = \lim_{x \to \infty} \frac{f_\infty(x)}{x^{p-1}} = \infty.
\]
Therefore, (H2)–(H3) hold for sufficiently small \( r \) and sufficiently large \( R \), respectively.

\end{example}
\subsection{Radial functional energy with two critical points on each direction}

In this subsection, we construct an explicit example of an operator \(T\) and a functional \(\mathcal{F}\) such that, for each direction $v$, the associated radial energy functional $\mathcal{E}(v)$ 
has two critical points: one corresponding to a global maximum point and the other to a local minimum. We remark that, for such a problem, a an approach similar to Theorem~\ref{pa1} is not applicable, since there we require that require the mapping \( t \mapsto \mathcal{E}(v)(t) \) has a global maximizer, which is the unique critical point (relation \eqref{unica zero}). However, the abstract result, Theorem~\ref{t1}, can still be applied, as it only requires \( \mathcal{E}(v) \) to have a unique maximum point.  

This example is illustrative since many results based on the Nehari manifold method typically require that, for each direction \( v \), the mapping \( t \mapsto E(tv) \) has a unique critical point, which is a global
maximum point  (see, e.g., \cite[Chapter 3]{nehari}, \cite[Theorem 2.1]{figueredo}). Therefore, our results can indeed be applied to a larger class of problems.


We consider the fixed point  problem
\begin{align}\label{problem aplicatie concreta}
u(t) =  \int_0^1 k(t,s)\,f(u(s))\, ds =: T(u)(t), \quad  t \in (0,1),
\end{align}
where the kernel \( k\colon [0,1]^2\to \mathbb{R}_+ \) is given by
\[
k(t,s) =
\begin{cases}
t(1-s), & t \leq s, \\
s(1-t), & s \leq t,
\end{cases}
\]
and $f(x)=a_2 x^2 + a_1 x + a_0$ with coefficients
\[
a_2 = 10^{-2}, \quad a_1 = 2.5\times 10^{-3}, \quad a_0 = 1.
\]
In order to apply the abstract result, let $X=C[0,1]$,  endowed with the supremum norm
\[
|u|_\infty = \max_{t\in [0,1]} \,|u(t)|,
\] and let  the
cone \( K \subset X \) be given by
\[
K = \left\{ u \in C[0,1]: u \geq 0 \ \text{and} \ \min_{t \in \left[\frac{1}{4}, \frac{3}{4}\right]} u(t) \geq \frac{1}{4} |u|_{\infty} \right\}.
\]
The functional \( \mathcal{F} : K \times K \times K_1\to \mathbb{R} \) is chosen to be
\[
\mathcal{F}(u,v,w) = \int_0^1 v(t)w(t) \, dt - \int_0^1 u(t) w(t) \, dt.
\]
Standard arguments (see, e.g., \cite{infante}) show that \( T \) is  completely continuous from $X$ to $X$. It follows immediately that the cone $K$ is invariant under the operator $T$, i.e., $T(K)\subset K$. Indeed, first note that since $a_0, a_1, a_2>0$ then for any $u\geq 0$ we have $T(u)\geq 0$. Next, let $u\in K$ and $\sigma\in \left[\frac{1}{4}, \frac{3}{4}\right].$ Since (see, e.g., \cite{infante}),\begin{equation*}
    k(t,s)\geq \frac{1}{4} \Phi(s)\,\,\, \text{ for all }t\in \left[\frac{1}{4}, \frac{3}{4}\right] \,\,\text{and}\,\,s\in [0,1],
\end{equation*}
and \begin{equation*}
    k(t,s)\leq  \Phi(s)\,\,\, \text{ for all }\,t,s\in  [0,1],
\end{equation*} where $\Phi(s)=s(1-s)$, we have \begin{align*}
    T(u)(\sigma)&=\int_{0}^ {1}k(\sigma,s)f(u(s))ds\\&\geq \frac{1}{4}\int_{0}^ {1}\Phi(s)f(u(s))ds\\&\geq \frac{1}{4}\int_{0}^ {1}k(t,s)f(u(s))ds\\&=T(u)(t),
\end{align*}
for all $t\in [0,1]$. Thus, $T(u)\in K$.

Next, in order to provide  an explicit expression for the radial energy functional \(\mathcal{E}\), we introduce the following notation
\[
\alpha_u(k) = \int_0^1 \int_0^1 k(t,s)\, u(s)^k \, u(t)\, ds \, dt,
\]
for all \( u \in K \) and \( k \geq 0 \).
Let $v\in K_1$ and $\sigma>0$. Then, we have
\begin{align}\label{aplicatie concreta F}
   \mathcal{F}(T(\sigma v), \sigma v,v) &= 
   \sigma|v|_{L^2}^2- \int_0^1 T(\sigma v)(t)\, v(t) \, dt\\&=
   \sigma|v|_{L^2}^2- \int_0^1 \left( \int_0^1 k(t,s) f(\sigma v(s)) \, ds \right) v(t) \, dt.\nonumber
    \\&
    =-a_2 \sigma^2\alpha_v(2)+\sigma\left( |v|_{L^2}^2-a_1 \alpha_v(1) \right)-a_0 \alpha_v(0).\nonumber
\end{align}
Therefore, from \eqref{aplicatie concreta F}, one has
\begin{align}\label{E aplicatie concreta}
    \mathcal{E}
(v)(t)&=
\int_0^t  \mathcal{F}(T(\sigma v), \sigma v,v)d\sigma
\\&=-a_2 \frac{t^3}{3} \alpha_v(2)+\frac{t^2}{2} \left( |v|_{L^2}^2-a_1 \alpha_v(1) \right)-a_0 \alpha_v(0)t.
\end{align}
In what follows, we show that conditions (h1)-(h3) of Theorem \ref{t1} are satisfied in the entire cone $K$, that is, \begin{equation*}
    r=0\quad \text{ and }\quad R=+\infty.
\end{equation*}

\textit{Check of (h1).} Let $v\in K_1.$
Then,  one has
\begin{align*}
    & |v|_{L^2}^2=\int_0^1 v(s)^2ds\geq \int_{1/4}^{3/4} v (s)^2ds\geq \int_{1/4}^{3/4} \left( \frac{1}{4}|v|_\infty\right)^2ds=\frac{1}{32},
    \\&
    |v|_{L^2}^2=\int_0^1 v(s)^2ds\leq |v|_{\infty}=1,\\&
    \alpha_v(k)=\int_0^1 \int_0^1 k(t,s)\, v(s)^k \, v(t)\,  ds \, dt\leq |v|_{\infty}^{k+1}\int_0^1 \int_0^1 k(t,s) ds \, dt =\frac{1}{12},
    \\&
        \alpha_v(k)=\int_0^1 \int_0^1 k(t,s)\,v(s)^k \,  v(t)\, ds \, dt\geq     \frac{1}{ 4^{k+1}}|u|_{\infty}^{k+1}\int_{1/4}^{3/4} \int_{1/4}^{3/4} k(t,s)ds \, dt=\frac{1}{32}\frac{1}{4^{k+1}}.
\end{align*}
From these, we can deduce the subsequent robust estimates
\begin{align}\label{estimari}
&a_2 \alpha_v(2) \in \left[ 4.88 \times 10^{-6}, \, 8.34 \times 10^{-4} \right], \\&
\nonumber |v|_{L^2}^2 - a_1 \alpha_v(1) \in \left[ 3 \times 10^{-2}, \, 1 \right], \\&
a_0 \alpha_v(0) \in \left[ 7.81 \times 10^{-3}, \, 8.34 \times 10^{-2} \right].\nonumber
\end{align}
Our next goal is to show that, for any choice of the parameters
\begin{align}\label{b}
    &  b_2 \in \left[4.88 \times 10^{-6}, \, 8.34 \times 10^{-4}\right], 
    \,\, 
    b_1 \in \left[3 \times 10^{-2}, \, 1  \right], 
    \\&
    b_0 \in \left[ 7.81 \times 10^{-3}, \, 8.34 \times 10^{-2} \nonumber \right],  
\end{align}
the cubic function
\[
   g(t) := -\,b_2 \frac{t^3}{3} + b_1 \frac{t^2}{2} - b_0 t
\]
admits exactly one maximum point on the positive semi-axis ($t>0$).
To prove this, we look for the critical points of $g$. Differentiating gives 
\[
g'(t) = -b_2 t^2 + b_1 t - b_0,
\]
so critical points are solutions of the quadratic equation 
\[
0 = -b_2 t^2 + b_1 t - b_0.
\]
The discriminant satisfies 
\begin{equation}\label{estimare discriminant}
b_1^2 - 4 b_2 b_0 \geq 9 \times 10 ^{-4}  
- 2.79 \times 10^{-4}
> 4 \cdot 10^{-4} > 0,
\end{equation}
hence there are two distinct positive roots,
\[
t_{\pm} = \frac{b_1 \pm \sqrt{b_1^2 - 4 b_2 b_0}}{2 b_2} > 0.
\]
Notice that by \eqref{b}  and \eqref{estimare discriminant}, we have\begin{equation}\label{estimare t+}
    t_+= \frac{b_1 + \sqrt{b_1^2 - 4 b_2 b_0}}{2 b_2}> \frac{3\times 10^{-2}}{1.7\times 10^{-3}}>10,
\end{equation}
and 
\begin{equation}\label{estimare t+ 2}
     t_+= \frac{b_1 + \sqrt{b_1^2 - 4 b_2 b_0}}{2 b_2}<\frac{1}{b_2}<10^7.
\end{equation}
One clearly has,
\[
g(0) = 0, 
\quad g(t) < 0 \ \text{for sufficiently small } t>0, 
\quad \text{and} \qquad g(t) \to -\infty \ \text{as } t \to \infty.
\]
Hence, if
\begin{equation}\label{g t+ si g t-}
    g(t_{+}) > 0 
    \quad \text{and} \quad 
    g(t_{-}) < 0,
\end{equation}
then \(t_{+}\) is the unique global maximizer, while \(t_{-}\) corresponds to a local minimizer.

To verify \eqref{g t+ si g t-}, observe that since  
\[
-b_{2} t_\pm^{2} + b_{1} t_\pm - b_{0} = 0,
\]
we have  
\begin{align*}
   g(t_\pm) 
   &= \tfrac{1}{3} t_\pm \left( -b_2 t_\pm^2 + \tfrac{3}{2} b_1 t_\pm - 3 b_0 \right) \\
   &= \tfrac{1}{3} t_\pm \left( -b_2 t_\pm^2 + b_1 t_\pm - b_0 + \tfrac{1}{2} b_1 t_\pm - 2 b_0 \right) \\
   &= \tfrac{1}{3} t_\pm \left( \tfrac{1}{2} b_1 t_\pm - 2 b_0 \right) \\
   &= \tfrac{1}{6} t_\pm \left( b_1 t_\pm - 4 b_0 \right).
\end{align*}
Therefore, establishing \(g(t_{+})>0\) and \(g(t_{-})<0\) reduces to proving  
\[
b_{1} t_{+} - 4 b_{0} > 0 
\quad \text{and} \quad 
b_{1} t_{-} - 4 b_{0} < 0,
\]
respectively. The first inequality is immediate, since by 
 \eqref{estimare discriminant}, we have
\begin{align*}
    b_1 t_+ - 4 b_0
&= \frac{1}{2 b_2} \Big( b_1^2 + b_1 \sqrt{b_1^2 - 4 b_2 b_0} - 8 b_0 b_2 \Big)\\&
\geq \frac{1}{2 b_2} \Big(9 \times 10^{-4} + 6 10^{-4} - 8 \times 8.34^2\times 10^{-6} \Big)\\&\geq
\frac{1}{2 b_2} \Big(15 \times 10^{-4} -5.57\times 10^{-4} \Big)>0.
\end{align*}
 For the second one,
 denote $\gamma=4b_0b_2.$
Then, \begin{align*}
     b_1 t_- - 4 b_0&=\frac{1}{2b_2}\left( b_1^2-b_1\sqrt{b_1^2-\gamma}-2\gamma \right)\\&
     =\frac{1}{2b_2}\left( b_1^2-\gamma -b_1\sqrt{b_1^2-\gamma}-\gamma \right)\\&=\frac{1}{2b_2}\left(\sqrt{b_1^2-\gamma}\left( \sqrt{b_1^2-\gamma}-b_1\right)-\gamma \right)\\&
     <0,
\end{align*}
where the latter inequality follows since $\gamma>0$ and $\sqrt{b_1^2-\gamma}-b_1<0$. Hence, \( g(t_-) < 0, \) as desired.

From the above results and given the uniform estimates in \eqref{estimari} (which are independent of \( v \in K_1\)), we deduce that for each $v\in K_1$, the mapping
\[
t \mapsto \mathcal{E}(v)(t),
\]
  has exactly one maximum point on the positive semi-axis, that is $t_v=t_+$, which by \eqref{estimare t+} and \eqref{estimare t+ 2} is bounded away from both zero and infinity, more exactly,  
$r_{0}=10$  and  $R_{0}=10^{7}.$
Hence, condition (h1) is verified.

\textit{Check of (h2)}. Note that for any nonnegative function \(u \in C[0,1]\), we have  
\begin{equation*}
    |T(u)|_\infty 
    \geq \max_{t \in [0,1]} a_{0} \int_{0}^{1} k(t,s)\,ds
    = a_{0} \max_{t \in [0,1]} \frac{t(1-t)}{2} 
    = \frac{a_{0}}{4}.
\end{equation*}
Since every \(u \in U_{b}\) is nonnegative, it follows that  
\[
\inf_{u \in U_{b}} |T(u)|_{\infty} \geq \frac{a_{0}}{4} > 0,
\]
and therefore condition (h2) is satisfied.

\textit{Check of (h3)}. Assume that 
\begin{equation*}
    \mathcal{F}\left(t u,u,\frac{u}{|u|_\infty}\right)=0,
\end{equation*}
for some $t>0$ and $%
u\in K\setminus \{0\}.$ Then, \begin{equation*}
    \frac{1}{|u|_\infty} \left( |u|^2_{L^2}-t |u|^2_{L^2}\right)=0,\end{equation*}
    which yields $t=1$, so (h3) holds.

    Consequently, applying Theorem \ref{t1}, we deduce that $T$ has a fixed point in $U_b$.
\section*{Declarations}

%
\textbf{Data availability} This manuscript has no associated data.
 
\subsection*{Conflict of interest}
The authors have no conflict of interest to declare.








\end{document}